\documentclass[12pt,reqno,twoside]{amsart}
\usepackage{amsthm}
\usepackage{booktabs}

\usepackage{amssymb}
\usepackage{graphics}   
\usepackage{tikz}
\usetikzlibrary{shapes,backgrounds,calc}
\usepackage{latexsym}
\usepackage{multicol}
\usepackage{verbatim,enumerate}
\usepackage{accents}
\usepackage{cite}
\usepackage{enumitem}
\usepackage{array}
\usepackage{mathtools}
\usepackage{dsfont}
\usepackage{nicematrix}

\usepackage[colorlinks=true, linkcolor=blue, citecolor=blue, urlcolor=blue]{hyperref}
\usepackage{hyperref}
\usepackage{amsmath, amscd,url}
\usetikzlibrary{decorations.pathreplacing}

\usepackage{setspace}

\usepackage{pstricks}

\advance\textwidth by 1.5in 
\usepackage{geometry}
\newtheorem{thm}{Theorem}[section]
\newtheorem{cor}[thm]{Corollary}
\newtheorem{lem}[thm]{Lemma}
\newtheorem{prop}[thm]{Proposition}
\theoremstyle{definition}

\newtheorem{defn}[thm]{Definition}

\newtheorem{obs}[thm]{Observation}
\numberwithin{equation}{subsection}

\newcounter{cnt}
 \makeatletter
\def\mydggeometry{\makeatletter\dg@YGRID=1\dg@XGRID=20\unitlength=0.003pt\makeatother}
\makeatother \theoremstyle{remark}

\numberwithin{equation}{section}
\let\bwdg\bigwedge
\def\bigwedge{{\textstyle\bwdg}}

\newcommand{\thmref}[1]{Theorem~\ref{#1}}

\newcommand{\lemref}[1]{Lemma~\ref{#1}}
\newcommand{\propref}[1]{Proposition~\ref{#1}}
\newcommand{\corref}[1]{Corollary~\ref{#1}}

\newcommand{\nc}{\newcommand}
\newcommand{\rnc}{\renewcommand}

\nc{\cal}{\mathcal} \nc{\goth}{\mathfrak} \rnc{\bold}{\mathbf}

\nc\bomega{{\mbox{\boldmath $\omega$}}} \nc\bpsi{{\mbox{\boldmath $\Psi$}}}
 \nc\balpha{{\mbox{\boldmath $\alpha$}}}
 \nc\bpi{{\mbox{\boldmath $\pi$}}}
 \nc\bvpi{{\mbox{\boldmath $\varpi$}}}
\nc\chara{\operatorname{ch}}

  \nc\bxi{{\mbox{\boldmath $\xi$}}}
\nc\bmu{{\mbox{\boldmath $\mu$}}} \nc\bcN{{\mbox{\boldmath $\cal{N}$}}} \nc\bcm{{\mbox{\boldmath $\cal{M}$}}} \nc\blambda{{\mbox{\boldmath
$\lambda$}}}\nc\bnu{{\mbox{\boldmath $\nu$}}}

\makeatletter
\def\section{\def\@secnumfont{\mdseries}\@startsection{section}{1}%
  \z@{.7\linespacing\@plus\linespacing}{.5\linespacing}%
  {\normalfont\scshape\centering}}
\def\subsection{\def\@secnumfont{\bfseries}\@startsection{subsection}{2}%
  {\parindent}{.5\linespacing\@plus.7\linespacing}{-.5em}%
  {\normalfont\bfseries}}
\makeatother

 \nc{\Hom}{\operatorname{Hom}}
  \nc{\mode}{\operatorname{mod}}
\nc{\End}{\operatorname{End}} \nc{\wh}[1]{\widehat{#1}} \nc{\Ext}{\operatorname{Ext}} \nc{\ch}{\text{ch}} \nc{\ev}{\operatorname{ev}}
\nc{\Ob}{\operatorname{Ob}} \nc{\soc}{\operatorname{soc}} \nc{\rad}{\operatorname{rad}} \nc{\head}{\operatorname{head}}

\def\det{\operatorname{det}}

 \nc{\Cal}{\cal} \nc{\Xp}[1]{X^+(#1)} \nc{\Xm}[1]{X^-(#1)}
\nc{\on}{\operatorname} \nc{\Z}{{\bold Z}} \nc{\J}{{\cal J}}  \nc{\Q}{{\bold Q}}

\nc{\N}{{\bold N}}  \nc\boa{\bold a} \nc\bob{\bold b} \nc\boc{\bold c} \nc\bod{\bold d} \nc\boe{\bold e} \nc\bof{\bold f} \nc\bog{\bold g}
\nc\boh{\bold h} \nc\boi{\bold i} \nc\boj{\bold j} \nc\bok{\bold k} \nc\bol{\bold l} \nc\bom{\bold m} \nc\bon{\mathbb n} \nc\boo{\bold o}
\nc\bop{\bold p} \nc\boq{\bold q} \nc\bor{\bold r} \nc\bos{\bold s} \nc\boT{\bold t} \nc\boF{\bold F} \nc\bou{\bold u} \nc\bov{\bold v}
\nc\bow{\bold w} \nc\boz{\bold z}\nc\ba{\bold A} \nc\bb{\bold B} \nc\bc{\mathbb C} \nc\bd{\bold D} \nc\be{\bold E} \nc\bg{\bold
G} \nc\bh{\bold H} \nc\bi{\bold I} \nc\bj{\bold J} \nc\bk{\bold K} \nc\bl{\bold L} \nc\bm{\bold M} \nc\bn{\mathbb N} \nc\bo{\bold O} \nc\bp{\bold
P} \nc\bq{\bold Q} \nc\br{\bold R} \nc\bs{\bold S} \nc\bt{\bold T} \nc\bu{\bold U} \nc\bv{\bold V} \nc\bw{\bold W} \nc\bz{\mathbb Z} \nc\bx{\bold
x} \nc\KR{\bold{KR}} \nc\rk{\bold{rk}} \nc\het{\text{ht }}

\nc\toa{\tilde a} \nc\tob{\tilde b} \nc\toc{\tilde c} \nc\tod{\tilde d} \nc\toe{\tilde e} \nc\tof{\tilde f} \nc\tog{\tilde g} \nc\toh{\tilde h}
\nc\toi{\tilde i} \nc\toj{\tilde j} \nc\tok{\tilde k} \nc\tol{\tilde l} \nc\tom{\tilde m} \nc\ton{\tilde n} \nc\too{\tilde o} \nc\toq{\tilde q}
\nc\tor{\tilde r} \nc\tos{\tilde s} \nc\toT{\tilde t} \nc\tou{\tilde u} \nc\tov{\tilde v} \nc\tow{\tilde w} \nc\toz{\tilde z} \nc\woi{w_{\omega_i}}

\newcommand\wrapped[1]%
{%
	\begin{array}{@{}l@{}}#1\end{array}%
}

\ifTUTeX
  \usepackage{fontspec}
\else
  \usepackage[T1]{fontenc}
  \usepackage[utf8]{inputenc} 
  \DeclareUnicodeCharacter{200B}{{\hskip 0pt}}
\fi

\begin{document}
	
	 \title[The Cycle Rank Threshold]{The Cycle Rank Threshold: Perfect Matchings and Bipartite Parter Graphs}

   \author{G. Arunkumar}
	
	\address{Indian Institute of Technology Madras, Chennai, India.}
	\email{garunkumar@iitm.ac.in}

   \author{Puja Samanta}
   \address{Indian Institute of Technology Madras, Chennai, India.}	\email{ma23d004@smail.iitm.ac.in, pujasamanta1999@gmail.com}





\begin{abstract}
A graph on \(n\) vertices is called a Parter graph if there exists a nonsingular symmetric matrix, whose nonzero off-diagonal entries correspond exactly to the edges of the graph, such that all of its principal submatrices of order \(n-1\) are singular. Previously, a graph satisfying this condition was said to have property~(P). It was proved that, for bipartite graphs of cycle rank at most \(3\), being a Parter graph is equivalent to the existence of a perfect matching~\cite{puja}. We extend this result to cycle rank \(4\), proving
that every bipartite graph of cycle rank at most \(4\) is a Parter graph if and only if it has a perfect matching. Furthermore, we show that this
bound is sharp by constructing, for every integer \(r\ge5\), a connected
balanced bipartite Parter graph of cycle rank \(r\) that has no perfect
matching.
\end{abstract}

\maketitle

\medskip

\noindent\textbf{Keywords:}
Parter vertex; perfect matching; cycle rank; bipartite graph; Hall's theorem.

\section{Introduction}

Throughout the paper, all graphs are assumed to be simple and all matrices
are assumed to be real and symmetric. We use the standard graph-theoretic
terminology of \cite{west}. For a symmetric matrix \(A=(a_{ij})\in\mathbb{R}^{n\times n}\), its
\emph{support graph}, denoted by \(G(A)\), is the graph with vertex set
\([n]=\{1,\ldots,n\}\) and edge set
\[
E(G(A))=\bigl\{ij:i\neq j,\ a_{ij}\neq0\bigr\}.
\]
Thus, the edges of \(G(A)\) are determined by the nonzero off-diagonal
entries of \(A\), while the diagonal entries do not affect \(G(A)\).
 For a graph \(G\) on \(n\) vertices, let
\[
S(G)=\{A\in\mathbb{R}^{n\times n}:A^\top=A,\; G(A)=G\},
\]
where \(A^\top\) denotes the transpose of \(A\).

An important question is how structural properties of \(G\) influence the spectral properties of matrices in \(S(G)\), particularly under vertex deletion.
The notion of a Parter vertex originates in the seminal work of
S. Parter~\cite{Parter} in 1960, who investigated how the structure of a tree
influences the multiplicities of eigenvalues of associated real symmetric
matrices. To define this concept, let \(A\in S(G)\), and for a vertex \(i\in V(G)\), let \(A(i)\)
denote the principal submatrix obtained by deleting the row and column
indexed by \(i\) (similarly, for a nonempty index set $\alpha\subseteq[n]$, let $A(\alpha)$ denote the principal submatrix obtained by deleting the rows and columns indexed by $\alpha$). If \(m_A(\lambda)\) denotes the multiplicity of an
eigenvalue \(\lambda\) of \(A\), then Cauchy's interlacing theorem \cite[Theorem 4.3.17]{inter}
implies that
\[
-1\le m_{A(i)}(\lambda)-m_A(\lambda)\le 1.
\]
A vertex \(i\in V(G)\) is called a \emph{Fiedler vertex} of \(A\) with
respect to an eigenvalue \(\lambda\) if
\(
m_{A(i)}(\lambda)\ge m_A(\lambda).
\)
When \(\lambda=0\), it is called an \emph{F-vertex} \cite{Fiedler}. A Fiedler vertex for
which
\[
m_{A(i)}(\lambda)=m_A(\lambda)+1
\]
is called a \emph{Parter vertex} with respect to \(\lambda\); in particular,
when \(\lambda=0\), it is called a \emph{P-vertex}~\cite{kim1}. A Fiedler
vertex for which
\(
m_{A(i)}(\lambda)=m_A(\lambda)
\)
is called a \emph{neutral vertex}, while a vertex satisfying
\(
m_{A(i)}(\lambda)=m_A(\lambda)-1
\)
is called a \emph{downer vertex}~\cite{hermitian}.
We denote by \(P_\nu(A)\) the number of
P-vertices of \(A\). A graph \(G\) of order \(n\) is said to have the \emph{full P-vertex property} if there exists a matrix \(A\in S(G)\) such that
\(
P_\nu(A)=n.
\)
Johnson and Sutton~\cite{hermitian} proved that every singular acyclic matrix of order \(n\) has at most \(n-2\) P-vertices. Kim and Shader~\cite{kim1} later showed that this bound does not hold for nonsingular acyclic matrices. In particular, they proved that the path graph \(P_n\) has the full P-vertex property if and only if \(n\) is even. Since then, the full P-vertex property has been extensively studied for several classes of graphs, including trees, cycles, and other graph families; see, for example, \cite{fons1,fons2,fons3,fons4,fons5,fons6,du1,du2,du3,du4,du5,cruz}.

Motivated by these developments, Howlader, Raickwade, and Sivakumar~\cite[Definition~1.1]{Howlader}
reformulated the full P-vertex property for nonsingular matrices, introducing
it under the terminology ``property (P)'', and investigated this concept for
unicyclic graphs. Since a P-vertex is precisely a Parter vertex with respect
to the eigenvalue \(0\), we adopt the following more descriptive terminology.
\begin{defn}[Parter graph]
A graph \(G\) is called a \emph{Parter graph} if there exists a nonsingular
matrix \(A\in S(G)\) such that every vertex of \(G\) is a P-vertex of \(A\).
Such a matrix \(A\) is called a \emph{Parter realization} of \(G\).
\end{defn}
Thus, a Parter graph is exactly a graph having property~(P) in the
terminology of \cite{Howlader,sharma,puja,APP}. Equivalently, \(G\) is a
Parter graph if there exists a nonsingular \(A\in S(G)\) such that
\(
\det A(v)=0\ \text{for every }v\in V(G),
\)
or, equivalently,
\[
\operatorname{diag}(A^{-1})=\mathbf 0.
\]
We use the terminology Parter graph throughout the remainder of the paper.

Recall that a perfect matching is a spanning forest in which every connected component is an edge.
Recently, Sharma and Panda~\cite{sharma} established a fundamental
connection between Parter graphs and perfect matchings in bipartite graphs.
\begin{thm}\cite[Theorem 3.1]{sharma} \label{bipartite perfect matching theorem}
Let $G$ be a connected bipartite graph. If $G$ has a perfect matching, then
$G$ is a Parter graph.
\end{thm}
By \cite[Lemma~2.1]{Howlader}, the property of being a Parter graph is componentwise. Since the existence of a perfect matching is also componentwise, the preceding theorem extends to arbitrary bipartite graphs. A graph that is both bipartite and a Parter graph will be referred to as a bipartite Parter graph.
The converse of Theorem~\ref{bipartite perfect matching theorem} is the
main question addressed in this paper. It is known that the converse holds
for bipartite graphs of cycle rank at most $3$.

\begin{thm}\cite[Theorem~3.8]{puja}\label{thm:cyclerank3}
Let $G$ be a bipartite graph with cycle rank $m(G)\le 3$. Then $G$ is a Parter graph if and only if $G$ has a perfect matching.
\end{thm}

We establish the corresponding result for the next cycle rank
case.

\begin{thm}\label{thm:cyclerank-4}
Let \(G\) be a bipartite graph with cycle rank \(m(G)=4\). Then \(G\) is a
Parter graph if and only if \(G\) has a perfect matching.
\end{thm}

The proof of Theorem~\ref{thm:cyclerank-4} relies on the following structural result describing how a Parter graph behaves when a bridge is deleted. Throughout this paper, \(G-x\) denotes the graph obtained from \(G\) by deleting the vertex \(x\) together with all edges incident to \(x\), \(G-\{x,y\}\) denotes the graph obtained by deleting the vertices \(x\) and \(y\), and \(G-xy\) denotes the graph obtained by deleting the edge \(xy\). It is known that if two Parter graphs are joined by a single edge, the resulting graph is also a Parter graph~\cite{sharma} (see Theorem~\ref{thm:bridge-preservation}). We formally establish the converse in the following proposition.
\begin{prop}\label{prop:bridge-property}
Let $G$ be a connected Parter graph, and let $xy$ be a bridge of $G$.
Let $G_1$ and $G_2$ be the components of $G-xy$, where
$x\in V(G_1)$ and $y\in V(G_2)$. Then one of the following holds:
\begin{enumerate}
\item both $G_1$ and $G_2$ are Parter graphs;
\item both $G_1-x$ and $G_2-y$ are Parter graphs.
\end{enumerate}
\end{prop}

Furthermore, in Section~\ref{sec:sharpness}, we show that the converse of Theorem~\ref{bipartite perfect matching theorem} does not hold in general: a connected bipartite Parter graph need not possess a perfect matching. In particular, we demonstrate that the cycle rank bound of \(4\) is sharp. We achieve this by introducing a family of connected balanced bipartite graphs, denoted \(G_{p,q}\), which lack a perfect matching. In \thmref{thm:general-family}, we explicitly construct a block matrix that serves as a Parter realization of \(G_{p,q}\). Building upon this algebraic construction, we establish that for \emph{every} integer \(r \ge 5\), there exists a connected balanced bipartite Parter graph of cycle rank \(r\) that has no perfect matching. Finally, we show that \(G_{3,2}\) is, up to isomorphism, the unique
counterexample of order \(10\) and cycle rank \(5\) to the converse.

\section{Preliminaries}

We begin by recalling some standard terminology and known results for the reader's convenience.
\subsection{Cycle rank and perfect matchings}

The \emph{cycle space} of a graph \(G\), denoted by \(\mathcal{C}(G)\), is the vector space over the field \(\mathbb{F}_2=\{0,1\}\) spanned by the edge sets of all cycles in \(G\) \cite[p.~23--24]{diestel}.

\begin{defn}[{\cite[p.~24]{diestel}, \cite[p.~39]{harary}}]
The \emph{cycle rank} (or \emph{cyclomatic number}) of a graph \(G\), denoted by \(m(G)\), is the dimension of its cycle space.
\end{defn}

The following well-known formula will be used throughout.

\begin{thm}\cite[Theorem~4.5(b)]{harary}\label{thm:CR}
For every graph \(G\),
\[
m(G)=|E(G)|-|V(G)|+k(G),
\]
where \(k(G)\) denotes the number of connected components of \(G\).
\end{thm}

We rely on the following standard characterization of perfect
matchings. For a bipartite
graph \(G(U,V;E)\) and a subset \(S\subseteq U\), let
\(
N(S):=\{v\in V:\exists\,u\in S\text{ such that }uv\in E\}
\)
denote the set of neighbors of \(S\). 

\begin{thm}[Hall's theorem, {\cite[Theorem~3.1.11]{west}}]\label{thm:hall}
A bipartite graph \(G(U,V;E)\), with \(|U|=|V|\), has a perfect matching
if and only if
\[
|N(S)|\ge |S|
\quad \text{for every }S\subseteq U.
\]
\end{thm} 

\subsection{Structural Results for Parter Graphs}

The following results describe structural properties of Parter graphs
concerning bipartitions and small orders. Note that these original sources used the term ``property (P)''; here, we use the terminology Parter graph for consistency.


\begin{thm}[{\cite[Theorem 2.3]{puja}}] \label{thm:balanced}
If a bipartite graph is a Parter graph, then it is balanced.
\end{thm}

For a graph \(G\), we write \(\delta(G)\) for its minimum degree.

\begin{lem}[{\cite[Lemma 3.6]{puja}}]\label{lem:isolated_vertex}
If a graph $G$ contains an isolated vertex, then $G$ is not a Parter graph. Consequently, any Parter graph must satisfy $\delta(G) \ge 1$.
\end{lem}

\begin{thm}[{\cite[Theorem 2.3]{Howlader}}]\label{lem:pendant_heredity}
Let $G$ be a graph. If $u$ is a pendant vertex of $G$ with the unique neighbor $v$, then $G - \{u,v\}$ is a Parter graph if and only if $G$ is a Parter graph.
\end{thm}

\begin{thm}[{\cite[Theorem 4.4]{puja}}]\label{8+2k}
Let $G$ be a bipartite graph of order $8+2k \ (k \ge 0)$ with at least $k$ pendant vertices. Then $G$ is a Parter graph if and only if $G$ has a perfect matching.
\end{thm}
The preceding results yield the following lower bound on the order of a
counterexample.
\begin{cor}[Minimum order of a counterexample]
\label{cor:min-order-counterexample}
If a connected bipartite Parter graph does not have a perfect matching, then its order is at least \(10\). Moreover, if \(G\) is of order \(10\), then \(\delta(G)\geq 2\).
\end{cor}

\begin{proof}
By \thmref{thm:balanced}, any bipartite Parter graph is balanced, and thus has even order. If \(|V(G)| \le 8\), applying Theorem~\ref{8+2k} with \(k=0\) implies that \(G\) must have a perfect matching, which contradicts our hypothesis. Therefore, \(|V(G)| \ge 10\).

Moreover, if \(|V(G)| = 10\), Theorem~\ref{8+2k} with \(k=1\) shows that \(G\) cannot have a pendant vertex. Consequently, \(\delta(G) \ge 2\).
\end{proof}

Let \(G \cdot H\) denote the graph obtained from two disjoint graphs
\(G\) and \(H\) by joining a vertex of \(G\) to a vertex of \(H\) by a
single edge.

\begin{thm}[{\cite[Theorem~5.3]{sharma}}]\label{thm:bridge-preservation}
Let \(G\) and \(H\) be two connected Parter graphs. Then \(G\cdot H\) is a Parter graph.
\end{thm}

\section{Bipartite Parter Graphs of Cycle Rank 4}

In this section, we prove the main theorem for bipartite graphs of cycle rank $4$, using \propref{prop:bridge-property} as the key structural result.

\begin{proof}[Proof of \propref{prop:bridge-property}]
Since $G$ is a Parter graph, there exists a nonsingular matrix $A\in S(G)$ such that
$$
    \det A(z)=0 \ \text{for every }z\in V(G).
$$
Ordering the vertices according to $V(G_1)$ and $V(G_2)$, we may write
$$
    A=
    \begin{pmatrix}
    A_1 & a e_xe_y^{\top}\\
    a e_ye_x^{\top} & A_2
    \end{pmatrix},
$$
where $A_i\in S(G_i)$, $a\neq 0$, and $e_x \in \mathbb{R}^{\vert{}V(G_1)\vert{}}$ and $e_y \in \mathbb{R}^{\vert{}V(G_2)\vert{}}$ are the standard basis vectors corresponding to the vertices $x$ and $y$, respectively. Because the only nonzero entries in the off-diagonal blocks are those corresponding to the bridge \(xy\), a direct Laplace expansion of \(\det A\) along the column corresponding to the vertex \(y\) yields
\begin{equation}\label{detA}
    \det A = \det(A_1)\det(A_2) - a^2\det(A_1(x))\det(A_2(y)).
\end{equation}
Furthermore,
\begin{equation}\label{detA(x)A(y)}
\det A(x)=\det(A_1(x))\det(A_2)
\quad\text{and}\quad
\det A(y)=\det(A_1)\det(A_2(y)).
\end{equation}
Since \(G\) is a Parter graph,
\[
\det A(x)=\det A(y)=0.
\]
Since \(\det A\neq0\), \(\det(A_1)\) and \(\det(A_2)\) must either both be zero or both be nonzero. For suppose exactly one is zero, say \(\det(A_1)=0\) and \(\det(A_2)\neq0\). Then \(\det A(x)=0\) implies \(\det(A_1(x))=0\), and hence
\[
\det A
=\det(A_1)\det(A_2)-a^2\det(A_1(x))\det(A_2(y))
=0,
\]
a contradiction. The case \(\det(A_1)\neq0\) and \(\det(A_2)=0\) follows similarly.

Therefore, exactly one of the following two cases occurs.

\medskip
\noindent
\noindent\textbf{Case 1.}
Suppose that \(\det(A_1)\neq0\) and \(\det(A_2)\neq0\).

It follows from \eqref{detA(x)A(y)} that
$$
    \det(A_1(x))=\det(A_2(y))=0.
$$
For $z\in V(G_1)\setminus\{x\}$, we have
$$
    0=\det A(z) =\det(A_1(z))\det(A_2) - a^2\det(A_1(z,x))\det(A_2(y)).
$$
Since $\det(A_2)\neq 0$ and $\det(A_2(y))=0$, it follows that
$$
    \det(A_1(z))=0.
$$
Thus, $A_1\in S(G_1)$ is a Parter realization of $G_1$. Similarly, $A_2\in S(G_2)$ is a Parter realization of $G_2$.

\medskip
\noindent
\textbf{Case 2.}
Suppose that \(\det(A_1)=\det(A_2)=0\).

Since $\det A\neq 0$, it follows from \eqref{detA} that
$$
    \det(A_1(x))\neq 0 \qquad\text{and}\qquad \det(A_2(y))\neq 0.
$$
For \(z\in V(G_1)\setminus\{x\}\),
\[
\det A(z)
=-a^2\det(A_1(z,x))\det(A_2(y))=0,
\]
Hence $$\det(A_1(z,x))=0.$$
Therefore, \(A_1(x)\in S(G_1-x)\) is a Parter realization of $G_1-x$. An analogous argument shows that \(G_2-y\) is also a Parter graph.
\end{proof}

We now prove the main result of this section.

\begin{proof}[Proof of Theorem~\ref{thm:cyclerank-4}]
The sufficiency follows from \thmref{bipartite perfect matching theorem}.
Conversely, suppose that $G$ is a Parter graph. We prove that $G$ has a perfect matching by induction on $|V(G)|$.

By \lemref{lem:isolated_vertex}, the graph $G$ has no isolated vertices. Hence,
$\delta(G)\ge 1.$
The result is immediate when $|V(G)|=2$.

Suppose that $G$ has a pendant vertex $u$, and let $v$ be its unique neighbor. The graph $G-\{u,v\}$ is also a Parter graph by \thmref{lem:pendant_heredity}. Moreover, deleting vertices cannot increase the cycle rank, so 
$$m(G-\{u,v\})\le m(G) = 4.$$
By Theorem~\ref{thm:cyclerank3}, if \(m(G-\{u,v\})\le 3\), or by the induction hypothesis if \(m(G-\{u,v\})=4\), the graph \(G-\{u,v\}\) has a perfect matching \(M'\). Consequently,
$$
    M'\cup\{uv\}
$$
is a perfect matching of $G$.

We may therefore assume that $\delta(G)\ge 2.$
Let \(G_1,G_2,\ldots,G_t\) be the connected components of \(G\).
By \cite[Lemma~2.1]{Howlader}, every connected component \(G_i\) is a Parter graph. Consequently, Theorem~\ref{thm:balanced} implies that each \(G_i\) is balanced.
Suppose, for a contradiction, that some component $H$ of $G$ has no perfect matching. Since the cycle rank is nonnegative and additive over connected components,
$$
    m(H)\le m(G) = 4.
$$
Let $(U,V)$ be the bipartition of $H$, where
$
    |U|=|V|=n.
$
By Hall's theorem, there exists $U_1\subseteq U$ such that
$
    |N(U_1)|<|U_1|,
$
where \(N(U_1)\) denotes the set of all neighbors of the vertices in \(U_1\).
Let
\[
V_1=N(U_1), \qquad
U_2=U\setminus U_1, \qquad
V_2=V\setminus V_1,
\]
and write \(|U_1|=r\), \(|V_1|=k\), so that
\(
d:=r-k\ge1.
\)

Since $\delta(G) \ge 2$ and $H$ is a component of $G$, we have $\delta(H) \ge 2$. 
Let \(e(X,Y)\) denote the number of edges with one endpoint in
\(X\) and the other in \(Y\). Since \(N(U_1)=V_1\), every vertex of \(U_1\) has
at least two neighbors in \(V_1\), and hence
\(
e(U_1,V_1)\ge2r.
\)
Moreover, there are no edges between \(U_1\) and \(V_2\). Consequently, every
vertex of \(V_2\) has all its neighbors in \(U_2\). As \(\delta(H)\ge2\), it
follows that
\(
e(U_2,V_2)\ge2(n-k).
\)
Finally, as \(H\) is connected, there exists an edge joining
\(U_1\cup V_1\) and \(U_2\cup V_2\). Since \(e(U_1,V_2)=0\), it follows that
\(
e(U_2,V_1)\ge1.
\)

Summing over these disjoint edge sets, we obtain
$$
\begin{aligned}
    |E(H)| &\ge e(U_1, V_1) + e(U_2, V_2) + e(U_2, V_1) \\
           &\ge 2r + 2(n-k) + 1 \\
           &= 2n + 2(r-k) + 1 \\
           &= 2n + 2d + 1.
\end{aligned}
$$
Since \(H\) is connected and \(m(H)\le4\), Theorem~\ref{thm:CR} gives
$$
    |E(H)|=|V(H)|-1+m(H) \le 2n+3.
$$
Combining these gives $$2n+3 \ge 2n+2d+1,$$ and hence $d \le 1$. Since \(d\ge1\), we conclude that \(d=1\).
This forces \(m(H)=4\). Hence equality holds throughout the preceding chain of inequalities. In particular,
\[
e(U_2,V_1)=1.
\]
Let \(xy\) be the unique edge between \(U_2\) and \(V_1\), where
\(x\in U_2\) and \(y\in V_1\). Since there are no edges between \(U_1\) and \(V_2\), the edge \(xy\) is a bridge of \(H\). Let \(H_1\) and \(H_2\) be the subgraphs of \(H\) induced by
\(U_1\cup V_1\) and \(U_2\cup V_2\), respectively. Then \(H_1\) and \(H_2\)
are the two components of \(H-xy\).
Since \(d=1\), we have
\[
|U_1|=|V_1|+1
\quad\text{and}\quad
|V_2|=|U_2|+1.
\]
Thus, neither \(H_1\) nor \(H_2\) is balanced.

Since \(xy\) is a bridge of the Parter graph \(H\), \propref{prop:bridge-property} implies that either both \(H_1\) and \(H_2\) are Parter graphs, or both \(H_1-y\) and \(H_2-x\) are Parter graphs.
The former is impossible, since \(H_1\) and \(H_2\) are unbalanced. The latter is also impossible, as deleting \(y\) from \(H_1\) and \(x\) from \(H_2\) increases the difference between the partite sets to \(2\), so both \(H_1-y\) and \(H_2-x\) are unbalanced. In either case, this contradicts Theorem~\ref{thm:balanced}, which requires all bipartite Parter graphs to be balanced.
Therefore, every connected component of \(G\) has a perfect matching. Consequently, \(G\) has a perfect matching.
\end{proof}

\section{Sharpness of the Cycle Rank Bound}\label{sec:sharpness}

We now show that the cycle rank bound in Theorem~\ref{thm:cyclerank-4} is sharp. To this end, we construct a family of connected balanced bipartite Parter graphs that lack a perfect matching. This construction also provides counterexamples to the converse of
Theorem~\ref{bipartite perfect matching theorem} with arbitrarily large cycle rank.

\subsection{A General Family of Counterexamples}
For integers \(p>q\ge2\), define the bipartite graph \(G_{p,q}\) as follows. Let
$
U=U_1\cup U_2$ and $V=V_1\cup V_2,
$
where
$$
|U_1|=|V_2|=p\ \text{and} \ |U_2|=|V_1|=q.
$$
Let
$
U_2=\{x_1,\ldots,x_q\}$ and 
$V_1=\{v_1,\ldots,v_q\}.$
The edge set of $G_{p,q}$ is given by
$$E(G_{p,q}) = \{uv : u \in U_1, v \in V_1\} \cup \{xw : x \in U_2, w \in V_2\} \cup \{x_i v_i\}_{i=1}^q.$$
This construction is illustrated in Figure~\ref{fig:general-family}.
\begin{figure}[ht]
\centering
\begin{tikzpicture}[
    vertex/.style={circle, fill=black, inner sep=2pt}
]

\node[vertex] (u1) at (-1.2,6) {};
\node[vertex] (u2) at (0,6) {};
\node[vertex] (u3) at (1.2,6) {};
\node at (2.1,6) {$\cdots$};
\node[vertex] (u4) at (3.0,6) {};
\node[vertex] (u5) at (4.2,6) {};
\node[vertex] (up) at (5.4,6) {};

\node at (-1.6,6) {$u_1$};
\node at (5.8,6) {$u_p$};

\draw[gray!40] (2.1,6) ellipse (4.2 and 0.65);

\node at (7,6) {$U_1$};

\node[vertex] (v1) at (0,4) {};
\node[vertex] (v2) at (1.2,4) {};
\node at (2.1,4) {$\cdots$};
\node[vertex] (v3) at (3,4) {};
\node[vertex] (vq) at (4.2,4) {};

\node at (-0.4,4) {$v_1$};
\node at (4.6,4) {$v_q$};

\draw[gray!40] (2.1,4) ellipse (3 and 0.55);

\node at (5.5,4) {$V_1$};

\foreach \u in {u1,u2,u3,u4,u5,up}
{
    \foreach \v in {v1,v2,v3,vq}
    {
        \draw (\u) -- (\v);
    }
}

\draw[
    decorate,
    decoration={brace,amplitude=6pt,mirror}
]
(-2.7,5.8) -- (-2.7,4.5)
node[midway,left=8pt] {$K_{p,q}$};

\node[vertex] (x1) at (0,2.0) {};
\node[vertex] (x2) at (1.2,2.0) {};
\node at (2.1,2.0) {$\cdots$};
\node[vertex] (x3) at (3.0,2.0) {};
\node[vertex] (xq) at (4.2,2.0) {};

\node at (-0.4,2.0) {$x_1$};
\node at (4.6,2.0) {$x_q$};

\draw[gray!40] (2.1,2.0) ellipse (3 and 0.55);

\node at (5.5,2.0) {$U_2$};

\draw (v1) -- (x1);
\draw (v2) -- (x2);
\draw (v3) -- (x3);
\draw (vq) -- (xq);


\node[vertex] (w1) at (-1.2,0) {};
\node[vertex] (w2) at (0,0) {};
\node[vertex] (w3) at (1.2,0) {};
\node at (2.1,0) {$\cdots$};
\node[vertex] (w4) at (3.0,0) {};
\node[vertex] (w5) at (4.2,0) {};
\node[vertex] (wp) at (5.4,0) {};

\node at (-1.6,0) {$w_1$};
\node at (5.8,0) {$w_p$};

\draw[gray!40] (2.1,0) ellipse (4.2 and 0.65);

\node at (7,0) {$V_2$};

\foreach \u in {x1,x2,x3,xq}
{
    \foreach \w in {w1,w2,w3,w4,w5,wp}
    {
        \draw (\u) -- (\w);
    }
}

\draw[
    decorate,
    decoration={brace,amplitude=6pt,mirror}
]
(-2.7,1.8) -- (-2.7,0.5)
node[midway,left=8pt] {$K_{p,q}$};

\end{tikzpicture}

\caption{The bipartite graph $G_{p,q}$.}
\label{fig:general-family}
\end{figure}
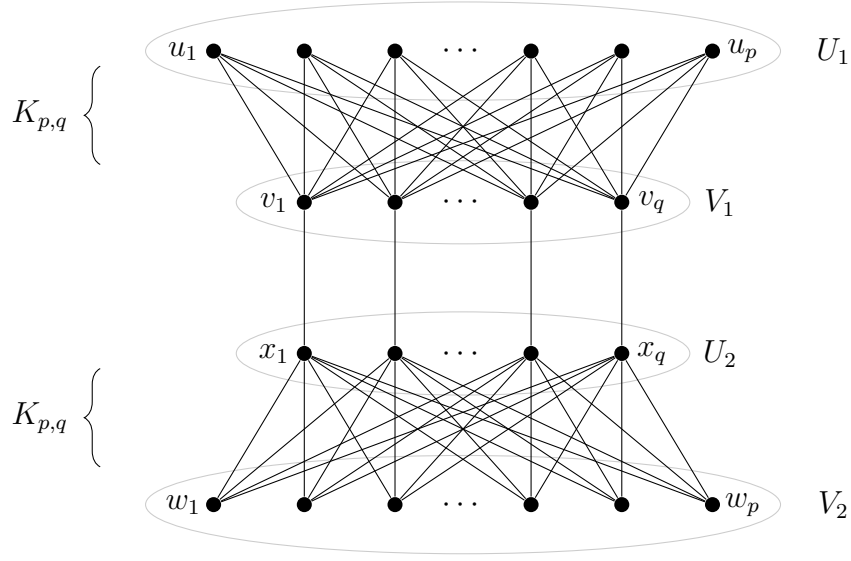

Thus, with respect to the ordered bipartitions
\((U_1,U_2)\) and \((V_1,V_2)\), the biadjacency matrix of
\(G_{p,q}\) is
$$
\begin{pmatrix}
J_{p,q}& \mathbf0\\
I_q&J_{q,p}
\end{pmatrix},
$$
where \(J_{a,b}\) denotes the \(a\times b\) all-ones matrix, \(\mathbf0\) denotes the \(p\times p\) zero matrix.
\begin{thm}\label{thm:general-family}
Let \(p>q\ge2\). Then \(G_{p,q}\) is a connected balanced bipartite Parter graph without a perfect matching, and has cycle rank
$$
m(G_{p,q})=(q-1)(2p-1).
$$
In particular, for every odd integer \(r\ge5\), there exists a connected
balanced bipartite Parter graph \(G\) with \(m(G)=r\) that does not have a perfect matching.
\end{thm}

\begin{proof}
By construction, \(G_{p,q}\) is connected and balanced. Moreover,
\(
N(U_1)=V_1.
\)
Since \(p>q\),
\(
|N(U_1)|<|U_1|.
\)
Thus, by Hall's theorem, \(G_{p,q}\) has no perfect matching.
The graph has
$
|V(G_{p,q})|=2(p+q)
$
vertices and
$
|E(G_{p,q})|=pq+q+pq=2pq+q
$
edges. Therefore,
$$
\begin{aligned}
m(G_{p,q})
&=|E(G_{p,q})|-|V(G_{p,q})|+1\\
&=2pq+q-2p-2q+1\\
&=(q-1)(2p-1).
\end{aligned}
$$
It remains to show that \(G_{p,q}\) is a Parter graph.
Set
$$
\lambda=\frac{pq(p-1)}{2p-q}
\quad
\text{and}
\quad
c^2=\frac{q(p-q)(p-q+1)}{2p-q}.
$$
Since $p>q\geq 2$, we have $\lambda>0$ and $c^2>0$. 

Choose $c>0$ satisfying the above equation.
Consider the matrix
$$
B=
\begin{pmatrix}
J_{p,q} & \mathbf0\\
cI_q & J_{q,p}
\end{pmatrix},
$$
and define
$$
A=
\begin{pmatrix}
\sqrt{\lambda}\,I_{p+q} & B\\
B^\top & \sqrt{\lambda}\,I_{p+q}
\end{pmatrix}.
$$
Clearly, $ A\in S(G_{p,q}).$
Set
$
R=\lambda I_{p+q}-BB^\top.
$

Since
$$
BB^\top=
\begin{pmatrix}
qJ_p & cJ_{p,q}\\
cJ_{q,p} & c^2I_q+pJ_q
\end{pmatrix},
$$
we obtain
$$
R=
\begin{pmatrix}
\lambda I_p-qJ_p & -cJ_{p,q}\\
-cJ_{q,p} & (\lambda-c^2)I_q-pJ_q
\end{pmatrix}.
$$
By the definitions of $\lambda$ and $c$,
$$
\lambda-c^2=q(q-1).
$$

Set
$$
\Delta=
\frac{pq^2(p-q+1)}{2p-q}.
$$
Using
$$
J_p^2=pJ_p,\qquad
J_q^2=qJ_q,\qquad
J_{p,q}J_{q,p}=qJ_p,\qquad
J_{q,p}J_{p,q}=pJ_q,
$$
a direct block multiplication verifies that
$$
R^{-1}
=
\begin{pmatrix}
\dfrac{1}{\lambda}(I_p-J_p)
&
\dfrac{c}{\Delta}J_{p,q}
\\[3mm]
\dfrac{c}{\Delta}J_{q,p}
&
\dfrac{1}{q(q-1)}(I_q-J_q)
\end{pmatrix}.
$$
In particular, $R$ is nonsingular and since \(I_p-J_p\) and \(I_q-J_q\) have zero diagonal,
$
\operatorname{diag}(R^{-1})=0.
$

Similarly, set
$
S=\lambda I_{p+q}-B^\top B.
$

Since
$$
B^\top B=
\begin{pmatrix}
c^2I_q+pJ_q & cJ_{q,p}\\
cJ_{p,q} & qJ_p
\end{pmatrix},
$$
we have
$$
S=
\begin{pmatrix}
(\lambda-c^2)I_q-pJ_q & -cJ_{q,p}\\
-cJ_{p,q} & \lambda I_p-qJ_p
\end{pmatrix}.
$$
Thus \(S\) is obtained from \(R\) by simultaneously interchanging the two block rows and the two block columns. Hence, for a suitable permutation matrix \(P\),
$$
S=P^\top RP.
$$
It follows that
$
\det S=\det R \ne 0
$
and
$
S^{-1}=P^\top R^{-1}P.
$
Consequently, $S$ is nonsingular and \(S^{-1}\) also has zero diagonal.


Since \(\lambda>0\), the block
\(\sqrt{\lambda}\,I_{p+q}\) is nonsingular. Using the Schur complement \cite[Section~0.8.5]{inter},
$$
\begin{aligned}
\det A
&=
\det\left(\sqrt{\lambda}\,I_{p+q}\right)
\det\left(
\sqrt{\lambda}\,I_{p+q}
-
B^\top
\left(\sqrt{\lambda}\,I_{p+q}\right)^{-1}
B
\right)\\
&=
\det\left(\lambda I_{p+q}-B^\top B\right)\\
&=\det S\\
&\neq 0.
\end{aligned}
$$
Thus $A$ is nonsingular.

By the inverse formula for partitioned matrices
\cite[Section~0.7.3]{inter}, the two diagonal blocks of
$A^{-1}$ are
$$
(A^{-1})[U]=
\left(
\sqrt{\lambda}I_{p+q}
-\frac{1}{\sqrt{\lambda}}BB^\top
\right)^{-1}\
=
\sqrt{\lambda}\,R^{-1}
$$
and
$$
(A^{-1})[V]=
\left(
\sqrt{\lambda}I_{p+q}
-\frac{1}{\sqrt{\lambda}}B^\top B
\right)^{-1}\
=
\sqrt{\lambda}\,S^{-1}.
$$
Since both $R^{-1}$ and $S^{-1}$ have zero diagonal,
$$
(A^{-1})_{ii}=0
\qquad
\text{for every }i\in V(G_{p,q}).
$$
Hence,
$
\det A(i)
=
0
$
for every $i\in V(G_{p,q}).
$
Since $A$ is nonsingular, every vertex of $G_{p,q}$ is a P-vertex
of $A$. Therefore, \(G_{p,q}\) is a Parter graph.

For the final assertion, take $q=2$. Then
$$
m(G_{p,2})=2p-1.
$$
As $p\geq 3$, this gives every odd integer $r\geq 5$.
\end{proof}

\subsection{The Smallest Counterexample}

We now apply the above construction to \(p=3\) and \(q=2\). The resulting graph \(G_{3,2}\) has order
\[
|V(G_{3,2})|=2(p+q)=10
\]
and
\[
m(G_{3,2})=(q-1)(2p-1)=5.
\]
 By Theorem~\ref{thm:general-family}, \(G_{3,2}\) is a Parter graph, but does not have a perfect matching (see Figure~\ref{fig:sharp-example}).

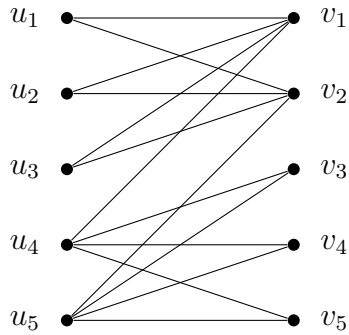
\begin{figure}[h]
\centering
\begin{tikzpicture}[scale=1]

\tikzset{
    vtx/.style={circle, fill=black, inner sep=1.6pt}
}

\node[vtx] (u1) at (0,4) {}; \node[left=6pt] at (u1) {$u_1$};
\node[vtx] (u2) at (0,3) {}; \node[left=6pt] at (u2) {$u_2$};
\node[vtx] (u3) at (0,2) {}; \node[left=6pt] at (u3) {$u_3$};
\node[vtx] (u4) at (0,1) {}; \node[left=6pt] at (u4) {$u_4$};
\node[vtx] (u5) at (0,0) {}; \node[left=6pt] at (u5) {$u_5$};

\node[vtx] (v1) at (3,4) {}; \node[right=6pt] at (v1) {$v_1$};
\node[vtx] (v2) at (3,3) {}; \node[right=6pt] at (v2) {$v_2$};
\node[vtx] (v3) at (3,2) {}; \node[right=6pt] at (v3) {$v_3$};
\node[vtx] (v4) at (3,1) {}; \node[right=6pt] at (v4) {$v_4$};
\node[vtx] (v5) at (3,0) {}; \node[right=6pt] at (v5) {$v_5$};


\draw (u1) -- (v1);
\draw (u1) -- (v2);

\draw (u2) -- (v1);
\draw (u2) -- (v2);

\draw (u3) -- (v1);
\draw (u3) -- (v2);

\draw (u4) -- (v3);
\draw (u4) -- (v4);
\draw (u4) -- (v5);

\draw (u5) -- (v3);
\draw (u5) -- (v4);
\draw (u5) -- (v5);

\draw (u4) -- (v1);
\draw (u5) -- (v2);

\end{tikzpicture}
\caption{The graph \(G_{3,2}\), a counterexample of cycle rank \(5\).}
\label{fig:sharp-example}
\end{figure}

Following the explicit construction in the proof of
Theorem~\ref{thm:general-family}, we have \(\lambda=3\) and \(c=1\).
The corresponding matrix \(A\in S(G_{3,2})\) is a Parter realization of
\(G_{3,2}\), and is given by
$$
A=
\begin{pmatrix}
\sqrt{3}I_5 & B\\[1mm]
B^\top & \sqrt{3}I_5
\end{pmatrix},
 \ 
 \text{where}
\
B=
\begin{pmatrix}
1&1&0&0&0\\
1&1&0&0&0\\
1&1&0&0&0\\
1&0&1&1&1\\
0&1&1&1&1
\end{pmatrix}.
$$

The next proposition identifies \(G_{3,2}\) as the unique smallest counterexample of cycle rank $5$
to the converse of Theorem~\ref{bipartite perfect matching theorem}.

\begin{prop}\label{prop:unique-order10-rank5}
Up to isomorphism, \(G_{3,2}\) is the unique connected balanced bipartite
graph of order \(10\) and cycle rank \(5\) that is a Parter graph but has no
perfect matching.
\end{prop}

\begin{proof}
Let $G$ be a connected balanced bipartite graph of order $10$ and
cycle rank $5$ that is a Parter graph but has no perfect matching.
Let $(U,V)$ be its bipartition, where
$
|U|=|V|=5.
$
By \corref{cor:min-order-counterexample}, $10$ is the
minimum possible order of such a graph, and any such graph of order \(10\) satisfies
\(
\delta(G)\geq2.
\)
Since $G$ has no perfect matching, Hall's theorem gives a set
$U_1\subseteq U$ satisfying
$
|N(U_1)|<|U_1|.
$

Set
$$
V_1=N(U_1),
\qquad
U_2=U\setminus U_1,
\qquad
V_2=V\setminus V_1.
$$

\noindent
Since $\delta(G)\geq2$, we have
$
|V_1|\geq2,
$
and consequently
$
|U_1|\geq3.
$

\noindent
On the other hand, there are no edges between $U_1$ and $V_2$.
Thus every vertex of $V_2$ has all its neighbors in $U_2$. Again,
since $\delta(G)\geq2$,
$
|U_2|\geq2.
$
Since $|U|=5$, this gives
$
|U_1|\leq3.
$
Therefore,
$
|U_1|=3.
$
Since
$$
2\leq|V_1|<|U_1|=3,
$$
we obtain
$
|V_1|=2.
$
Consequently,
$
|U_2|=2$
 and $|V_2|=3.
$

Each vertex of $U_1$ has degree at least $2$, and all its neighbors
belong to the two-element set $V_1$. Hence every vertex of $U_1$ is
adjacent to both vertices of $V_1$. Therefore,
$
G[U_1,V_1]\cong K_{3,2}.
$
Similarly, every vertex of $V_2$ has degree at least $2$, and all its
neighbors belong to the two-element set $U_2$. Hence
$
G[U_2,V_2]\cong K_{2,3}.
$

Since $G$ is connected and has cycle rank $5$,
$$
|E(G)|
=
m(G)+|V(G)|-1
=
5+10-1
=
14.
$$
Thus, besides the \(12\) edges within these two complete bipartite
subgraphs, there are exactly two additional edges. Since there are no
edges between \(U_1\) and \(V_2\), these two edges must join \(U_2\)
to \(V_1\).
Up to isomorphism, there are only two possibilities for these two
edges: either they are independent, or they have a common endpoint.
We consider these two cases separately.

In the first case, after relabelling the vertices, the two edges may
be taken to be $u_4v_1$ and $u_5v_2.$
This is precisely the graph \(G_{3,2}\).

It remains to rule out the second case, illustrated in Figure~\ref{fig:second-case}. Up to relabelling and
interchanging the two partite sets, suppose that the two additional
edges are $u_4v_1$ and $u_4v_2.$
Removing $u_4$ leaves exactly two
components,
$
K_{3,2}$
and
$K_{1,3}.
$

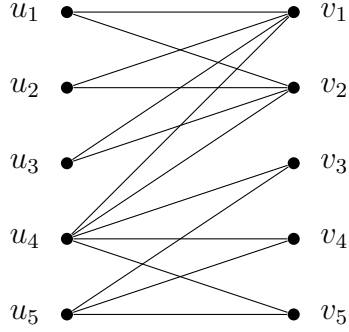
\begin{figure}[h]
\centering
\begin{tikzpicture}[scale=1]

\tikzset{
    vtx/.style={circle, fill=black, inner sep=1.6pt}
}

\node[vtx] (u1) at (0,4) {}; \node[left=6pt] at (u1) {$u_1$};
\node[vtx] (u2) at (0,3) {}; \node[left=6pt] at (u2) {$u_2$};
\node[vtx] (u3) at (0,2) {}; \node[left=6pt] at (u3) {$u_3$};
\node[vtx] (u4) at (0,1) {}; \node[left=6pt] at (u4) {$u_4$};
\node[vtx] (u5) at (0,0) {}; \node[left=6pt] at (u5) {$u_5$};

\node[vtx] (v1) at (3,4) {}; \node[right=6pt] at (v1) {$v_1$};
\node[vtx] (v2) at (3,3) {}; \node[right=6pt] at (v2) {$v_2$};
\node[vtx] (v3) at (3,2) {}; \node[right=6pt] at (v3) {$v_3$};
\node[vtx] (v4) at (3,1) {}; \node[right=6pt] at (v4) {$v_4$};
\node[vtx] (v5) at (3,0) {}; \node[right=6pt] at (v5) {$v_5$};


\draw (u1) -- (v1);
\draw (u1) -- (v2);

\draw (u2) -- (v1);
\draw (u2) -- (v2);

\draw (u3) -- (v1);
\draw (u3) -- (v2);

\draw (u4) -- (v3);
\draw (u4) -- (v4);
\draw (u4) -- (v5);

\draw (u5) -- (v3);
\draw (u5) -- (v4);
\draw (u5) -- (v5);

\draw (u4) -- (v1);
\draw (u4) -- (v2);

\end{tikzpicture}
\caption{The second possible configuration for the additional edges.}
\label{fig:second-case}
\end{figure}

Suppose, for a contradiction, that $G$ is a Parter graph. Let $A\in S(G)$ be a nonsingular matrix for which every vertex is a P-vertex. With respect to the decomposition
at the cut vertex $u_4$, write
$$
A=
\begin{pmatrix}
a&x^\top&y^\top\\
x&B&0\\
y&0&C
\end{pmatrix},
$$
where the first row and column correspond to $u_4$, $B\in S(K_{3,2})$, and $C\in S(K_{1,3})$.

Since $u_4$ is a P-vertex and $A$ is nonsingular,
$
\operatorname{nullity} A(u_4)=1.
$
But
$
A(u_4)=B\oplus C.
$
Hence exactly one of $B$ and $C$ is singular, and the other is
nonsingular.

Suppose first that $B$ is singular and $C$ is nonsingular. 
Applying the Schur complement with respect to
\(C\) gives
$$
\det A
=
\det(C)
\det
\begin{pmatrix}
\alpha & x^\top\\
x&B
\end{pmatrix},
$$
where \(\alpha = a - y^\top C^{-1}y\).
Using the bordered determinant identity (see, e.g., \cite[Chapter~0, p.~26]{inter}),
$$
\det
\begin{pmatrix}
\alpha&x^\top\\
x&B
\end{pmatrix}
=
\alpha\det B-x^\top\operatorname{adj}(B)x,
$$

and since \(B\) is singular, we obtain
$$
\det A
=
-\det(C)\,
x^\top\operatorname{adj}(B)x.
$$

Since $A$ and $C$ are nonsingular,
$
x^\top\operatorname{adj}(B)x\neq0.
$

For any vertex $z$ of the component corresponding to $C$,
$$
\det A(z)
=
-\det C(z)\,
x^\top\operatorname{adj}(B)x.
$$
Since every vertex is a P-vertex,
$\det A(z)=0,$
and therefore $\det C(z)=0$
for every vertex $z$ of $K_{1,3}$. Thus, the nonsingular matrix $C$
is a Parter realization of $K_{1,3}$, implying that $K_{1,3}$ is a Parter graph.
This is impossible by Theorem~\ref{thm:balanced}, since
$K_{1,3}$ is not balanced. 

If instead $C$ is singular and $B$ is nonsingular, the same argument
shows that $B$ would be a Parter realization of $K_{3,2}$,
which is again impossible by
Theorem~\ref{thm:balanced}.

Thus, the two additional edges cannot have a common endpoint.
Consequently, they must be independent. Hence, up to isomorphism,
$G$ is precisely the graph \(G_{3,2}\).
\end{proof}

\subsection{Counterexamples at Arbitrary Cycle Rank \(\ge5\)}

We first record the following observation.

\begin{obs}\label{obs:bridge-cycle-rank}
Let $G_1$ and $G_2$ be disjoint graphs, and let $G$ be obtained by
joining a vertex of $G_1$ to a vertex of $G_2$ by a single edge. Then
$$
m(G)=m(G_1)+m(G_2),
$$
which follows immediately from the cycle rank formula
$m(G)=|E(G)|-|V(G)|+k(G)$.
\end{obs}

\begin{lem}\label{lem:prescribed-cycle-rank-pm}
For every integer $s\ge0$, there exists a connected balanced bipartite
graph $H_s$ with a perfect matching and cycle rank $m(H_s)=s$.
\end{lem}

\begin{proof}
For $s=0$, take $H_0\cong K_2$. For $s\ge1$, take $s$ disjoint copies
of $C_4$ and join consecutive copies by bridges, choosing the endpoints of each bridge from opposite partite sets. The resulting graph $H_s$ is
connected, balanced, and bipartite. Moreover, the union of perfect
matchings in the individual copies of $C_4$ is a perfect matching of
$H_s$. Since $m(C_4)=1$, repeated application of
Observation~\ref{obs:bridge-cycle-rank} gives
$
m(H_s)=s.
$
\end{proof}

We now apply the observation to the smallest counterexample \(G_{3,2}\).

\begin{cor}\label{cor:all-cycle-ranks}
For every integer \(r\ge5\), there exists a connected balanced bipartite
graph \(G\) with
\(
m(G)=r
\)
such that \(G\) is a Parter graph but does not have a perfect
matching.
\end{cor}

\begin{proof}
By Theorem~\ref{thm:general-family}, the graph
$
B:=G_{3,2}
$
is a connected balanced bipartite Parter graph with no perfect
matching and
$
m(B)=5.
$
Thus, the case $r=5$ follows immediately.

Now let $r\ge6$, and let $H_{r-5}$ be the graph given by
Lemma~\ref{lem:prescribed-cycle-rank-pm}. Form $G$ from the disjoint
union of $B$ and $H_{r-5}$ by adding a single bridge between
vertices in opposite partite sets. Then $G$ is connected, balanced, and bipartite. Since
$H_{r-5}$ has a perfect matching, it is a Parter graph by
Theorem~\ref{bipartite perfect matching theorem}. Thus, both \(B\) and
\(H_{r-5}\) are Parter graphs. By
Theorem~\ref{thm:bridge-preservation}, \(G\) is also
a Parter graph. Moreover, by Observation~\ref{obs:bridge-cycle-rank}, 
$$
m(G)
=
m(B)+m(H_{r-5})
=
5+(r-5)
=
r.
$$
It remains to show that $G$ has no perfect matching. Let $e$ denote
the bridge between $B$ and $H_{r-5}$, and suppose that $M$ is a perfect
matching of $G$. If $e\notin M$, then the restriction of $M$ to $B$
is a perfect matching of $B$, a contradiction. If $e\in M$, then one
vertex of $B$ is matched across the bridge, leaving the remaining nine
vertices of $B$ to be matched among themselves, which is impossible.
Thus $G$ has no perfect matching.
\end{proof}

The construction in Corollary~\ref{cor:all-cycle-ranks} is illustrated in
Figure~\ref{fig:arbitrary-cycle-rank}.

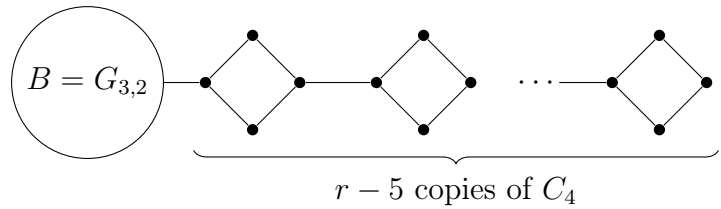
\begin{figure}[ht]
\centering
\begin{tikzpicture}[scale=0.78]

\node[
    circle,
    draw,
    minimum size=2.0cm
] (B) at (0,0) {$B=G_{3,2}$};

\node[circle, fill, inner sep=1.5pt] (a1) at (2.0,0) {};
\node[circle, fill, inner sep=1.5pt] (a2) at (2.8,0.8) {};
\node[circle, fill, inner sep=1.5pt] (a3) at (3.6,0) {};
\node[circle, fill, inner sep=1.5pt] (a4) at (2.8,-0.8) {};

\draw (a1)--(a2)--(a3)--(a4)--(a1);

\draw (B.east)--(a1);

\node[circle, fill, inner sep=1.5pt] (b1) at (4.9,0) {};
\node[circle, fill, inner sep=1.5pt] (b2) at (5.7,0.8) {};
\node[circle, fill, inner sep=1.5pt] (b3) at (6.5,0) {};
\node[circle, fill, inner sep=1.5pt] (b4) at (5.7,-0.8) {};

\draw (b1)--(b2)--(b3)--(b4)--(b1);

\draw (a3)--(b1);

\node at (7.65,0) {$\cdots$};

\node[circle, fill, inner sep=1.5pt] (c1) at (8.9,0) {};
\node[circle, fill, inner sep=1.5pt] (c2) at (9.7,0.8) {};
\node[circle, fill, inner sep=1.5pt] (c3) at (10.5,0) {};
\node[circle, fill, inner sep=1.5pt] (c4) at (9.7,-0.8) {};

\draw (c1)--(c2)--(c3)--(c4)--(c1);

\draw (8.0,0)--(c1);

\draw[
    decorate,
    decoration={brace, amplitude=5pt, mirror}
]
(1.8,-1.15) -- (10.7,-1.15)
node[midway,below=6pt] {$r-5\text{ copies of }C_4$};

\end{tikzpicture}
\caption{A counterexample with cycle rank \(r\ge5\), obtained by joining
\(B=G_{3,2}\) to \(r-5\) copies of \(C_4\) by bridges.}
\label{fig:arbitrary-cycle-rank}
\end{figure}

\section{Concluding Remarks}

We have established \(4\) as the sharp cycle rank threshold for the
equivalence between bipartite Parter graphs and perfect matchings. A
natural direction for future work is to further understand the structure
of bipartite Parter graphs of higher cycle rank, particularly those that
do not have a perfect matching.

\section*{Declaration of competing interest}
There is no competing interest.

\section*{Acknowledgements}

\noindent
The authors would like to thank Dr. Nishad Kothari for his valuable suggestions. The first author's research was supported by the Anusandhan National Research Foundation (ANRF), Government of India, under MATRICS Grant ANRF/ARGM/2025 /002777/MTR. The second author's research was supported by the National Board for Higher Mathematics (NBHM), Government of India, under the Ph.D. Scholarship Grant 0203/8(28)/2023-R\&D-II.

\end{document}